\documentclass[12pt]{amsart}
\usepackage[latin1]{inputenc}
\usepackage{color}
\usepackage{enumerate}
\usepackage{amssymb}
\usepackage[all]{xy}
\usepackage{hyperref}
\usepackage{tikz-cd}

\newtheorem{thm}{Theorem}[section]
\newtheorem{cor}[thm]{Corollary}
\newtheorem{lem}[thm]{Lemma}

\theoremstyle{definition}
\newtheorem{defn}[thm]{Definition}

\theoremstyle{remark}
\newtheorem{rem}[thm]{Remark}

\numberwithin{equation}{section}

\newcommand{\To}{\longrightarrow}

\newcommand{\cC}{\protect{\mbox{\sf Cat}}}
\newcommand{\cB}{\protect{\mbox{\sf Banach}}}
\newcommand{\cM}{\protect{\mbox{\sf Metric}}}
\newcommand{\cL}{\protect{\mbox{\sf BLattices}}}

\newcommand{\R}{\mathbb{R}}

\newcommand{\Q}{\mathbb{Q}}

\newcommand{\Lip}{{\mathrm{Lip}}_0}

\DeclareMathOperator{\dens}{dens}

\DeclareMathOperator{\Obj}{Obj}

\usepackage{color}
\usepackage{enumerate}
\begin{document}
\setcounter{tocdepth}{1}


\title[Local complementation and free constructions]{Local complementation in Banach spaces and its preservation under free constructions}

\author[Avil\'es]{Antonio Avil\'es}
\address[Avil\'es]{Universidad de Murcia, Departamento de Matem\'{a}ticas, Campus de Espinardo 30100 Murcia, Spain
	\newline
	\href{https://orcid.org/0000-0003-0291-3113}{ORCID: \texttt{0000-0003-0291-3113} } }
\email{\texttt{avileslo@um.es}}

\author[Mart\'inez-Cervantes]{Gonzalo Mart\'inez-Cervantes}
\address[Mart\'inez-Cervantes]{Universidad de Murcia, Departamento de Matem\'{a}ticas, Campus de Espinardo 30100 Murcia, Spain
	\newline
\href{http://orcid.org/0000-0002-5927-5215}{ORCID: \texttt{0000-0002-5927-5215} } }	
\email{gonzalo.martinez2@um.es}

\author[Rueda Zoca]{Abraham Rueda Zoca}
\address[Rueda Zoca]{Universidad de Granada, Facultad de Ciencias. Departamento de An\'alisis Matem\'atico, 18071 Granada, Spain
	\newline
	\href{https://orcid.org/0000-0003-0718-1353}{ORCID: \texttt{0000-0003-0718-1353} }}
\email{\texttt{abrahamrueda@ugr.es}}
\urladdr{\url{https://arzenglish.wordpress.com}}

\thanks{This work was partially supported by grants 21955/PI/22, funded by Fundaci\'on S\'eneca - ACyT Regi\'on de Murcia;
	PID2021-122126NB-C32 and PID2021-122126NB-C31 (Rueda Zoca), funded by MCIN/AEI/10.13039/501100011033 and FEDER A way of making Europe;
	CIGE/2022/9 (Martínez-Cervantes and Rueda Zoca), funded by Generalitat Valenciana  and FQM-0185 (Rueda Zoca), funded by Junta de Andaluc\'ia.}




\begin{abstract} 
In this work we consider natural generalizations of local complementation in Banach spaces, which include Lipschitz-local complementation.
We show that all these notions are indeed equivalent to 
the classical notion of local complementation of Banach spaces. As an application, we show that local complementation is naturally preserved under certain free constructions in Functional Analysis, including Lipschitz-free spaces and free Banach lattices.
\end{abstract}

\maketitle

\section{Introduction}

The idea of freely generated structures is widespread in different fields of mathematics. Ignoring technicalities for a moment, the basic idea is that if we have two categories $\cC_2 \subset \cC_1$, and an object $M$ of $\cC_1$, a free object generated by $M$ in $\cC_2$ is an object $F[M]$ together with a morphism $\delta\colon M \To F[M]$ in $\cC_1$ with the following universal property:
For every object $X$ in $\cC_2$ and every morphism $f\colon M\To X$ in $\cC_1$ there exists a unique morphism $T_f\colon F[M]\To X$ in $\cC_2$ such that $T_f\circ \delta = f$.\\

The most elementary example is when $\cC_1$ is the category of sets, and $\cC_2$ is the category of vector spaces. A free vector space generated by a set $M$ is just a vector space that contains $M$ as a basis. The free universal property expresses the idea that linear maps on a vector space are determined by arbitrary maps defined on the basis.\\

In this paper we are interested in two free constructions in the realm of functional analysis. One of them is when considering $\cC_1$ to be the category of pointed metric spaces and Lipschitz functions and $\cC_2$ the category of Banach spaces and bounded linear operators. The so-called Lipschitz-free Banach spaces $\mathcal{F}(M)$ have been a popular subject of research after the seminal paper of Godefroy and Kalton \cite{GK}, cf.~\cite{godesurv}. The second example is when $\cC_1$ is the category of Banach spaces and $\cC_2$ is the category of Banach lattices, a construction introduced in \cite{ART18} generalizing \cite{dPW15}. In all these categories, morphisms $f$ have a natural norm $\|f\|$ that is required to be preserved in the definitions of free objects.\\

When free objects exist, they are easily seen to be essentially unique. After choosing a particular $F[M]$ for every $M$, one has a functor that associates to every morphism $f\colon N\To M$ in $\cC_1$ the morphism $F[f]\colon F[N]\To F[M]$ in $\cC_2$ that is characterized by $F[f]\circ \delta_N = \delta_M\circ f$. A natural question is what kind of properties of $f$ are transferred to $F[f]$. For example, if $f$ is an inclusion (i.e.~an isomorphism onto its range), is $F[f]$ also an inclusion? Informally speaking, does $N\subset M$ imply $F[N]\subset F[M]$? In the case of Lipschitz-free spaces, the answer is affirmative when $\mathbb R$ is the scalar field because $\mathcal{F}(M)^*=\Lip(M)$ and because every Lipschitz-function $f\colon N\longrightarrow\mathbb R$ extends to a Lipschitz function $\tilde{f}\colon M\longrightarrow \mathbb R$ with the same Lipschitz norm \cite[Theorem 1.33]{weaver}. In the case of free Banach lattices, this happens to be a delicate question recently addressed by Oikhberg, Taylor, Tradacete and Troitsky \cite{OTTT}. 

A stronger property than being an inclusion is having a left inverse. This property is obviously preserved: If $g\circ f = id_M$, then $F[g] \circ F[f] = id_{F[M]}$. In the Banach space setting, having a left inverse means that the range is a complemented subspace. In this paper we investigate the preservation of local complementation under free constructions. This is an important weakening of complementation, arising from the \emph{local theory} of Banach spaces, where we are concerned with what happens at the level of finite-dimensional subspaces controlled by a uniform constant. 
 
\begin{defn}\label{defloccomp} Let $Y\subset X$ be Banach spaces and $\lambda\geq 1$. We say that
	$Y$ is locally $\lambda$-complemented in $X$ if for every finite-dimensional subspace $V$ of $X$ and every $\varepsilon>0$ there exists an operator $T\colon V\longrightarrow Y$ so that $T(v)=v$ for every $v\in V\cap Y$ and $\Vert T\Vert\leq \lambda+\varepsilon$.
\end{defn}

Perhaps the most remarkable fact related to this notion is the principle of local reflexivity that asserts that every Banach space is locally 1-complemented in its bidual (due to Lindenstrauss and Rosenthal \cite{linros}, see e.g. \cite[Theorem 9.15]{checos}). A more general fact is that local $\lambda$-complementation can be characterized by the existence of an operator $S:X\To Y^{**}$ of norm at most $\lambda$ that extends the canonical inclusion $Y\To Y^{**}$, also by the existence of a right inverse for the dual operator $X^*\To Y^*$, and also by the existence of $\lambda$-bounded extensions of operators into dual Banach spaces: every $T\colon Y\To Z^{*}$ extends to $\tilde{T}\colon X\To Z^*$ with $\|\tilde{T}\|\leq \lambda\|T\|$. We refer to the works of Fakhoury~\cite{Fakhoury} and Kalton~\cite{Kalton84} as fundamental early references on the theory. An open problem posed by Kalton \cite{Kalton08} is whether local complementation is equivalent to the existence of a Lipschitz retraction $r\colon X\To Y$. We consider a natural notion of Lipschitz local complementation (and some variants) and we show that it is equivalent to usual local complementation, cf. Theorem~\ref{LemmaLocalCompFreeBanach} and Theorem~\ref{Liplocomp}.

Another main result (Corollary~\ref{theo:LocCompFree}) is that local complementation is preserved by suitable free functors: If $N$ is a locally complemented subspace of a Banach space $M$, then $F[N]$ is locally complemented in $F[M]$. This gives in particular a sufficient condition for the preservation of isomorphic embeddings. A more precise characterization of when $F[N]$ is locally complemented in $F[M]$ is given in Theorem~\ref{TheoLocCompFreeObjects}. These results apply to the free Banach lattice generated by a Banach space as well as to the Lipschitz-free space of a Banach space. In the latter case, we are recovering  a result of Fakhoury \cite{Fakhoury} that for Banach spaces $N\subset M$, the adjoint $M^*\To N^*$ admits a right inverse operator if and only if the restriction between the spaces of Lipschitz maps $\Lip(M)\To \Lip(N)$ admits a right inverse operator. Remember that $\Lip(X) = \mathcal{F}(X)^*$.

We will treat the different free constructions from a unified perspective, by considering two categories $\cC_1$ and $\cC_2$ with some hypotheses. The category $\cC_2$ can be the Banach spaces or the Banach lattices but also the $p$-convex Banach lattices for which a free functor has been studied in \cite{jjttt}. For the category $\cC_1$ only Banach spaces and metric spaces with Lipschitz functions correspond to cases existing in the literature, but it could be any intermediate category like metric spaces with a scalar multiplication, or things of the like. 
Despite our efforts to make $\cC_1$ general enough to include the class of distributive lattices, in order to apply our main results to free Banach lattices generated by a lattice (see \cite{ARA18} for a definition), it turned out that this class requires a different treatment since several difficulties arise; we refer the reader to \cite{AMCRARZ21} for a detailed study of this case.


Another important result on local complementation is the Theorem of Sims and Yost \cite[Lemma III.4.4]{hww}: given any Banach space $X$ and any subspace $Y$ of $X$ there exists a locally complemented subspace $Z$ of $X$ such that $Y\subseteq Z$ and $\dens(Z)=\dens(Y)$, where $\dens$ denotes the density character. An analogous fact in the Lipschitz category has recently been proved in \cite{HQ}. A consequence of our results is that, under the hypotheses of Theorem \ref{TheoLocCompFreeObjects}, objects of the form $F[X]$ with $X$ a nonseparable Banach space have plenty of locally complemented subspaces of the form $F[Y]$ with $Y\subset X$ separable.

\section{Categorical framework}

We first discuss how to set a unified framework to treat Lipschitz-free spaces and free Banach lattices at once. A purely categorical does not fit our needs because we must impose some technical requirements that are too clumsy to be expressed in a purely categorical language. Instead, we incorporate a little bit of universal algebra or model theory inspired by the idea of continuous logic, cf.~\cite{continuouslogic}. Marcel de Jeu \cite{Marcel2020} also gave a general context for treating free objects in vector lattices, Banach lattices and related categories. We will not enter into deep waters with this, we will just introduce in a self-contained way a notion that is good for us, that we have called a metric-algebraic category and is rather simple. 

For a category $\cC$, we denote by $\Obj(\cC)$ the class of objects in $\cC$. A metric-algebraic category $\cC$ is one whose objects and morphisms can be described in the following way.

There must be an index set $I$ and natural numbers $\{n_i : i\in I\}$ such that every object of the category $\cC$ is of the form $(X,d,0,\{\sigma_i\}_{i\in I})$, where $d$ is a metric on the set $X$, $0$ is a fixed element of $X$ which we call the zero point of $X$ and each $\sigma_i$ is an $n_i$-ary operation on $X$, i.e.~a function of the form $\sigma_i\colon X^{n_i} \longrightarrow X$. The morphisms between $(X,d,0,\{\sigma_i\}_{i\in I})$ and $(Y,d',0',\{\tau_i\}_{i\in I})$ are the functions $T\colon X\To Y$ that are Lipschitz and preserve the zero point and the operations. That is,
$$\|T\| := \sup\left\{\frac{d'(Tx,Ty)}{d(x,y)} : x,y\in X, x\neq y\right\} <+\infty, \quad T(0)=0' \quad \text{ and }$$
$$T(\sigma_i(x_1,\ldots,x_{n_i})) = \tau_i(Tx_1,\ldots,Tx_{n_i}) \quad \text{ for all} \ (x_1,\ldots,x_{n_i})\in X^{n_i}.$$

Notice that, by the first property, a morphism $T$ in such a  category always has an associated norm $\|T\|$. One example is the category $\cM$ of metric spaces with a distinguished point $0$ and Lipschitz functions preserving this point. Notice that in this case the objects are of the form $(M,d,0)$ and no $n$-ary operation is considered.

Banach spaces constitute also a metric-algebraic category $\cB$, whose objects are described as
$$(X,d,0,\{l_{\alpha,\beta}\}_{\alpha,\beta\in\mathbb{R}}),$$
where $d$ is the distance induced by the norm of the Banach space $X$ and $l_{\alpha,\beta}:X \times X \longrightarrow X$ represents the $2$-ary operation which sends each pair $(x,y)$ to $\alpha x+\beta y$. By adding the lattice operations $\vee$ and $\wedge$, Banach lattices also constitute a metric-algebraic category $\cL$, whose morphisms are the bounded operators that commute with $\vee$ and $\wedge$. We will often abuse notation and simply denote an object $(X,d,0,\{\sigma_i\}_{i\in I})$ in a metric-algebraic category by $X$ for short.

\bigskip
If $\cC_1$ and $\cC_2$ are two metric-algebraic categories with index sets $I_1$ and $I_2$ respectively, then the inclusion $\cC_2 \subseteq \cC_1$ will mean that the list of operations that defines $\cC_2$ contains the list that defines $\cC_1$, and that every object of $\cC_2$ becomes an object of $\cC_1$ when removing the extra operations, i.e.~ $I_1 \subseteq I_2$ and if $(X,d,0,\{\sigma_i\}_{i\in I_2})$ is an object in $\cC_2$
then $(X,d,0,\{\sigma_i\}_{i\in I_1})$ belongs to $\cC_1$.

For example,
$$\cL \subseteq \cB \subseteq \cM .$$

Some further examples are the category $\texttt{Wedges}$ of wedges in Banach spaces (subsets closed under addition and nonnegative scalar multiplication) with the additive and positively homogeneous Lipschitz maps as morphisms, the category $\texttt{Convex}_0$ of convex subsets of Banach spaces that contain $0$ with Lipschitz maps that preserve convex combinations and $0$ as morphisms, and the category $\mathbb{Q}\texttt{-Convex}_0$ similar to the previous one but considering only  rational convex combinations. Notice that in the case of $\texttt{Wedges}$ we consider the $2$-ary operations $l_{\alpha,\beta}$ defined as before but only for $\alpha,\beta \geq 0$, whereas in $\texttt{Convex}_0$ and $\mathbb{Q}\texttt{-Convex}_0$ we consider the same operations $l_{\alpha,\beta}$ but with the restrictions that $0\leq \alpha \leq 1$ and $\beta=1-\alpha$ (and $\alpha \in \Q$ in the second case). Thus,
$$\cB \subseteq \texttt{Wedges} \subseteq \texttt{Convex}_0 \subset \mathbb{Q}\texttt{-Convex}_0\subseteq \cM .$$


In our context, the notion of free object can be formulated as follows: 

\begin{defn}
	Let $\cC_1, \cC_2$ be two metric-algebraic categories with $\cC_2 \subseteq \cC_1$. Let $X$ be an object in $\cC_1$. A \textit{free object} generated by $X$ in $\cC_2$ is a pair $(\delta_X, F[X])$, where $F[X] \in \Obj(\cC_2)$ and $\delta_X \colon X \longrightarrow F[X]$ is a morphism in $\cC_1$ with $\|\delta_X\|=1$ and the property that for every $Z \in \cC_2$ and every morphism $f\colon X\longrightarrow Z$ in $\cC_1$ there exists a unique morphism $T_f \colon F[X] \longrightarrow Z$ in $\cC_2$ such that $\|T_f\| = \|f\|$ and $T_f\circ \delta_X = f$. 
	$$\xymatrix{X \ar_{\delta_X}[d]\ar[rr]^f&& Z\\
		F[X]\ar_{T_f}[urr]&& }$$
\end{defn}

Notice that we have added the conditions $\|\delta_X\|=1$ and $\|T_f\| = \|f\|$ to the `purely categorical' definition of a free object; see, for instance, \cite[Definition 2.1]{Marcel2020}.
When one deals with the purely categorical definition then the free object generated by a given object in $\cC_1$ is unique up to isomorphisms in $\cC_2$ (see, e.g., the comment below \cite[Definition 2.1]{Marcel2020}). Furthermore, as the referee pointed out to us, the main result Theorem \ref{TheoLocCompFreeObjects} still holds true under the purely categorical definition. Nevertheless, in order to avoid creating incoherence in the literature and, in particular, with the classical definitions of the free Banach lattice over a Banach space and the Lipschitz-free space generated by a metric space, we opted for keeping the reference to the norm which, in addition, yields a unique object up to isometries in $\cC_2$:

\begin{lem}	Let $\cC_1, \cC_2$ be two metric-algebraic categories with $\cC_2 \subseteq \cC_1$ and let $X$ be an object in $\cC_1$. If we have two free objects $(\delta_X,F[X])$ and $(\delta'_X,F'[X])$ in $\cC_2$ generated by $X$ then there exists an isometric isomorphism in $\cC_2$ between $F[X]$ and $F'[X]$.
\end{lem}
\begin{proof}
The definition of free object applied to the functions $\delta_X$ and $\delta'_X$ yields two morphisms  $T_{\delta'_{X}}\colon F[X] \longrightarrow F'[X]$ and $T_{\delta_{X}}\colon F'[X] \longrightarrow F[X]$ in $\cC_2$ such that $\|T_{\delta'_{X}}\|=\|T_{\delta_{X}}\|=\|\delta_{X}\|=\|\delta'_{X}\|=1$ and $T_{\delta_{X}} \circ \delta'_X = \delta_X$ and $T_{\delta'_{X}} \circ \delta_X = \delta'_X$.
Thus, $T_{\delta_{X}} \circ T_{\delta'_{X}} \circ \delta_X = \delta_X.$
Furthermore,
$$ 1=\| \delta_X\| = \| T_{\delta_{X}} \circ T_{\delta'_{X}} \circ \delta_X\| \leq  \| T_{\delta_{X}} \circ T_{\delta'_{X}} \|\| \delta_X\|=\| T_{\delta_{X}} \circ T_{\delta'_{X}} \| \leq \|T_{\delta_{X}} \|\| T_{\delta'_{X}} \|=1,
$$
which implies that $\|T_{\delta_{X}} \circ T_{\delta'_{X}} \|=1$. Since the identity map $I\colon F'[X] \longrightarrow F'[X]$ is a $\cC_2$-morphism which also satisfies $\|I\|=\|\delta_X\|=1$ and $I\circ \delta_X =\delta_X$, the uniqueness of such morphism yields that $T_{\delta_{X}} \circ T_{\delta'_{X}}$ is the identity on $F'[X]$. The same argument shows that $T_{\delta'_{X}} \circ T_{\delta_{X}}$ is the identity on $F[X]$.
Since $\|T_{\delta_{X}}\|=\|T_{\delta'_{X}}\|=1$, we conclude that $F[X]$ and $F'[X]$ are isometrically isomorphic in $\cC_2$ as desired.
\end{proof}


A minor technical issue with this definition is that does not apply to the case $X=0$, because in that case one cannot have $\|\delta_X\|=1$.

If we choose a representative $(\delta_X,F[X])$ for each $X$ that has a free object, then the association $X\mapsto F[X]$ defines a functor: if $f\colon Y\longrightarrow X$ is a morphism in $\cC_1$, then $f$ induces a morphism $F[f]:=T_{\delta_X \circ f}:  	F[Y] \longrightarrow F[X]$ in $\cC_2$, characterized by the fact that $F[f]\circ \delta_Y = \delta_X\circ f$ and $\|F[f]\|=\|\delta_X \circ f \|$.

\section{Local complementation relative to a category $\cC_1$}

We now look at local complementation in a more general framework, when we relativize to some metric-algebraic category $\cC_1\supseteq \cB$. When $\cC_1= \cB$, what Theorem~\ref{LemmaLocalCompFreeBanach} states are the different characterizations of local complementation in Banach spaces mentioned in the introduction. The other case that we have in mind is when $\cC_1 = \cM$, so that morphisms become Lipschitz maps instead of operators. But Theorem~\ref{LemmaLocalCompFreeBanach} also applies to categories like $\texttt{Wedges}$, $\texttt{Conv}_0$ or $\mathbb{Q}$-$\texttt{Conv}_0$ mentioned earlier.
Although each class $\cC$ yields a natural notion of being
$\cC$-locally complemented, Theorem \ref{Liplocomp} shows that all these notions are equivalent to being (Banach) locally complemented.

We need an extra hypothesis. A metric-algebraic category $\cC_1$ is called \textit{nice} if whenever $f\colon A\To B$ is a morphism in $\cC_1$ and $D\subset B$ is an object of $\cC_1$ with the operations, the zero point and the metric inherited from $B$, then $f^{-1}(D)$ is also an object of $\cC_1$ when endowed with the operations, the zero point and the metric inherited from $A$, so in particular $f|_{f^{-1}(D)}\colon f^{-1}(D) \longrightarrow D$ is a morphism in $\cC_1$.

As a remark concerning notation, the canonical inclusion of a Banach space into its bidual is always denoted by the letter $j$.

\begin{thm}
	\label{LemmaLocalCompFreeBanach}
	Let $\cC_1$ be a nice metric-algebraic category with $\cB\subseteq \cC_1$. Let $X$ be a Banach space, $Y\subseteq X$ be a subspace and $\lambda \geq 1$. The following assertions are equivalent.
	
	\begin{enumerate}
		\item For every dual Banach space $Z^*$ and every morphism $f\colon Y\longrightarrow Z^*$ in $\cC_1$ there exists a morphism $h\colon X\longrightarrow Z^*$ in $\cC_1$ extending $f$ with $\|h\|\leq \lambda \|f\|$.
		\item There is a morphism $h\colon X\longrightarrow Y^{**}$ in $\cC_1$ with $\|h\|\leq \lambda $ extending the canonical inclusion $j\colon Y\longrightarrow Y^{**}$.
	
		\item For every $\varepsilon>0$ and every finite set $G\subseteq X$, there exists a set $V$ with $G\subset V\subset X$, so that $V$ is an object of $\cC_1$ with the operations, the metric and the zero point inherited from $X$  and there exists a morphism $\varphi:V\longrightarrow Y$ in $\cC_1$ with $\|\varphi\|\leq \lambda (1+\varepsilon)$ so that
		$\varphi(e)=e$ holds for every $e\in V\cap Y$.
		
		\item For every $\varepsilon>0$ and every finite set $G\subseteq X$, there exists a set $V$ with $G\subset V\subset X$, so that $V$ is an object of $\cC_1$ with the operations, the metric and the zero point inherited from $X$  and there exists a morphism $\varphi:V\longrightarrow Y$ in $\cC_1$ with $\|\varphi\|\leq \lambda (1+\varepsilon)$ so that
		$\|\varphi(e)-e\|<\varepsilon$ holds for every $e\in V\cap Y$.
		
	\end{enumerate}
\end{thm}

\begin{proof}
	
	(1)$\Rightarrow$(2). It is immediate taking $Z^*=Y^{**}$ and $f=j$.
	
	(2)$\Rightarrow$(3).  By the Principle of Local Reflexivity \cite[Theorem 9.15]{checos} there exists an operator $P\colon \text{span}(h(G))\longrightarrow Y$ so that
	\begin{itemize}
		\item $P(e)=e$ holds for every $e\in \text{span}(h(G))\cap Y$ and,
		\item $\Vert P\Vert\leq 1+\varepsilon$.
	\end{itemize}
	Since $\cC_1$ is nice, $V := h^{-1}(\text{span}(h(G)))$ is an object in $\cC_1$ with the operations, the zero point and the metric inherited from $X$ (as an object of $\cC_1$). Define $\varphi \colon V\longrightarrow Y$ by the formula $\varphi(v) = P(h(v))$. This is a morphisms in $\cC_1$ since $h$ is so and $P$ is an operator (remember that $\cB\subseteq \cC_1$). Given $e\in V\cap Y$, we have $h(e)=j(e)$, and this implies that  
	$$\varphi(e)=P(h(e))=P(j(e))=e.$$
	Moreover, by definition of $\varphi$ we have $$\|\varphi\| \leq \|P\|\cdot \|h\|\leq  (1+\varepsilon) \lambda.$$
	
	(3)$\Rightarrow$(4) is obvious. (4)$\Rightarrow$(1). Let
	$$\mathcal G:=\{(G,\varepsilon): G\subseteq X\mbox{ is finite and }0<\varepsilon<1\},$$
	which is a directed set under the order
	$$(G,\varepsilon)\leq (H,\delta)\Longleftrightarrow G\subseteq H\mbox{ and }\delta\leq \varepsilon.$$
	For every $(G,\varepsilon)\in\mathcal{G}$, our hypothesis provides a set $V_{(G,\varepsilon)}$ such that  $G\subseteq V_{(G,\varepsilon)} \subseteq X$ and  $V_{(G,\varepsilon)}$ is an object of $\cC_1$ when endowed with the operations, the metric and the zero point inherited from $X$, and a morphism $\varphi_{(G,\varepsilon)}\colon V_{(G,\varepsilon)} \longrightarrow Y$ in $\cC_1$  so that $\|\varphi_{(G,\varepsilon)}(e)-e\|<\varepsilon$ holds for every $e\in V_{(G,\varepsilon)}\cap Y$ and $\|\varphi_{(G,\varepsilon)}\|\leq \lambda (1+\varepsilon)$. 
	
	Now, given any dual Banach space $Z^*$ and any morphism $f\colon Y\longrightarrow Z^{*}$ in $\cC_1$, define for every $(G,\varepsilon)\in\mathcal G$ the $\cC_1$-morphism $f_{(G,\varepsilon)}:=f\circ \varphi_{(G,\varepsilon)}\colon  V_{(G,\varepsilon)}\longrightarrow Z^{*}$. Define $h_{(G,\varepsilon)}\colon X\longrightarrow Z^{*}$ by the equation
	$$h_{(G,\varepsilon)}(x):=\left\{\begin{array}{cc}
		f_{(G,\varepsilon)}(x)& \mbox{ if }x\in V_{(G,\varepsilon)},\\
		0 & \mbox{ if }x\notin V_{(G,\varepsilon)}.
	\end{array} \right.$$
	Note that $\|h_{(G,\varepsilon)}(x)\| \leq \|f_{(G,\varepsilon)}\|\|x\| \leq \|f\|\|\varphi_{(G,\varepsilon)}\| \|x\|\leq \lambda (1+\varepsilon)\|f\|\|x\|\leq 2\lambda \|f\|\|x\|$ for every $x \in X$. Therefore,
	$(h_{(G,\varepsilon)})_{(G,\varepsilon)\in\mathcal G}$ defines a net in the compact space
	$$\prod\limits_{x\in X} (2\lambda\|f\| \|x\| B_{Z^{*}}, w^*).$$

	We take a cluster point $h\colon X\rightarrow Z^*$ of this net. 
	
	\bigskip
	\textbf{Claim 1.} $h$ extends $f$.
	
	Pick $y\in Y$ and $\varepsilon>0$. Take $G=\{y\}$. For all $(H,\delta)\geq (G,\varepsilon)$ we have $$\|h_{(H,\delta)}(y) -f(y)\| = \|f_{(H,\delta)}(y) - f(y)\| = \|f(\varphi_{(H,\delta)}(y))- f(y))\| \leq \|f\| \|\varphi_{(H,\delta)}(y)- y\| \leq \|f\|\varepsilon.$$ We conclude that $h(y) = f(y)$.

	\bigskip
	\textbf{Claim 2.} $h$ is a Lipschitz function with $\|h\| \leq \lambda \|f\| $.
	
	Fix $x,y\in X$ and $\varepsilon>0$. If we consider $G = \{x,y\}$, then for every $(H,\delta) \geq (G,\varepsilon)$ we have
	$$ \|h_{(H,\delta)}(x) - h_{(H,\delta)}(y)\| = \|f_{(H,\delta)}(x) - f_{(H,\delta)}(y)\| \leq \|f\|\|\varphi_{(H,\delta)}\|\|x-y\| \leq \lambda (1+\varepsilon)\|f\|\|x-y\|.$$
	We must be careful here taking limits because the norm is not continuous in the weak$^*$ topology. But if we take any $z\in B_Z$ we have 
	$$ |\left(h_{(H,\delta)}(x) - h_{(H,\delta)}(y)\right)(z)| \leq \lambda (1+\varepsilon)\|f\|\|x-y\|,$$
	for all $(H,\delta)\geq (G,\varepsilon)$. And from this we get that
	$$ |\left(h(x) - h(y)\right)(z)| \leq \lambda (1+\varepsilon)\|f\|\|x-y\|$$
	for all $z\in B_Z$ and all $\varepsilon$, so we get $\|h(x)-h(y)\|\leq \lambda\|f\|\|x-y\|$ as desired.

	\bigskip
	
	\textbf{Claim 3.} $h$ is a morphism in $\cC_1$.
	
	We already showed that $h$ is Lipschitz, so it only remains to check that $h$ preserves the operations that define category $\cC_1$. Remember that we are assuming $\cB\subseteq \cC_1$, so we have to check that $h$ preserves some kind of linear combination. Namely, given $\alpha,\beta \in \R$, we must show that
	$$ h(\alpha x + \beta y)=\alpha h(x)+\beta h(y) $$
	for every $x,y\in X$.
	We fix $x,y\in X$, $G:=\{x,y, \alpha x+\beta y\}$ and some $\varepsilon>0$. For all $(H,\delta)\geq (G,\varepsilon)$ we have
	\[\begin{split}
	h_{(H,\delta)}(\alpha x +\beta y) & = f_{(H,\delta)}(\alpha x +\beta y) \\
	&  = \alpha f_{(H,\delta)}(x) + \beta f_{(H,\delta)}(y)\\
	&  = \alpha h_{(H,\delta)}(x)+\beta  h_{(H,\delta)}(y).\end{split}\]
	Since linear combinations are weak$^*$-continuous, we get the desired result.
\end{proof}

When $\cC_1 = \mathtt{Banach}$, condition (3) above is equivalent to the fact that the Banach space $Y$ is $\lambda$-complemented in $X$ (Definition \ref{defloccomp}). Thus, it is natural to introduce the following definition:
When $Y\subseteq X$ satisfy the assumptions of Theorem~\ref{LemmaLocalCompFreeBanach}, we will say that $Y$ is $\cC_1$-locally $\lambda$-complemented in $X$, and we will say that $Y$ is $\cC_1$-locally complemented in $X$ when there is $\lambda\geq 1$ so that $Y$ is $\cC_1$-locally $\lambda$-complemented in $X$.

 By looking at condition (2), we observe that if $\cC_1 \subseteq \cC'_1$, then there are fewer morphisms in $\cC_1$ and therefore $\cC_1$-locally $\lambda$-complementation implies  $\cC'_1$-locally $\lambda$-complementation. The strongest possible notion is $\cB$-local complementation (just called local complementation of Banach spaces) while the weakest notion is $\cM$-local complementation, when the functions involved are just Lipschitz. But there is actually no point in making all these distinctions because of the following result.

\begin{thm}\label{Liplocomp}
	For $Y\subseteq X$ Banach spaces and $\lambda\geq 1$, $Y$ is locally $\lambda$-complemented in $X$ if and only if $Y$ is Lipschitz locally $\lambda$-complemented in $X$. Therefore, all notions of $\cC_1$-local complementation are equivalent, for all nice metric-algebraic categories $\cC_1\supseteq \cB$.
\end{thm}

\begin{proof}
We assume that $Y$ is $\cM$-locally $\lambda$-complemented, and we must prove that $Y$ is $\cB$-locally $\lambda$-complemented in $X$. Let $\Lip(X)$ be the Banach space of all Lipschitz functions $f\colon X\To\mathbb{R}$ that vanish at 0, supplied with the Lipschitz norm. By \cite[Proposition 2.11]{Kalton08}, there exists an operator $P\colon \Lip(X)\To X^*$ with $\|P\|=1$	and such that $P(f)|_Y = f|_Y$ holds for every $f\in \Lip(X)$ with $f|_Y$ linear. By Theorem~\ref{LemmaLocalCompFreeBanach}(2) we have a $\lambda$-Lipschitz function $h\colon X\To Y^{**}$ that extends the canonical inclusion $Y\subset Y^{**}$. Define $\phi\colon Y^{***}\To \Lip(X)$ as $\phi(y^{***})(x) = y^{***}(h(x))$. First of all, we check that $\phi$ is an operator with $\|\phi\|\leq \lambda$ whose range is in fact inside $\Lip(X)$. Linearity is trivial, and
$$\|\phi(y^{***})(x) - \phi(y^{***})(z)\| = \|y^{***}(h(x)-h(z))\| \leq \|y^{***}\|\lambda\|x-z\| $$
which shows that $\|\phi(y^{***})\|\leq \lambda \|y^{***}\|$. Now, we consider $\psi:Y^* \To X^*$ given by $\psi = P\circ \phi|_{Y^*}$, which is an operator with $\|\psi\|\leq \lambda$. We claim that $\psi^*|_{X}: X \To Y^{**}$ is the operator from Theorem~\ref{LemmaLocalCompFreeBanach}(2) that we are looking for. We just need to check that $\psi^*(y) = j(y)$ for all $y\in Y$. That is, $\psi^*(y)(y^\ast) = y^\ast(y)$ for all $y^\ast\in Y^*$. But
$$\psi^\ast(y)(y^\ast) = \psi(y^\ast)(y) = P(\phi(y^\ast))(y) = P(e_{y^\ast}\circ h)(y),$$
where $e_{y^\ast}:Y^{**}\To \mathbb{R}$ is the evaluation function on $y^{\ast}$.	Notice that $(e_{y^{*}}\circ h)(z) = y^{\ast}(z)$ for all $z\in Y$, because $h$ extends the canonical inclusion $Y\subset Y^{**}$. In particular, $(e_{y^{*}}\circ h)|_Y$ is linear, so by the properties of $P$, $P(e_{y^{*}}\circ h)|_Y = (e_{y^{*}}\circ h)|_Y $, so $P(e_{y^\ast}\circ h)(y) = y^{*}(y)$ and this concludes the proof.
\end{proof}	

\begin{rem}\label{rem:mejorakalton}
The previous theorem improves \cite[Proposition 3.21]{Kalton08}, where it is proved that, given two Banach spaces $Y\subset X$, the existence of a Lipschitz retraction $r \colon X\longrightarrow Y$ implies that $Y$ is locally complemented in $X$.
\end{rem}

\section{Local complementation under a free functor}

Let $\cC_2\subseteq \cB$. Given a Banach space $Z\in \Obj(\cC_2)$, we can consider its bidual space $Z^{**}$. If we can extend the operations in $Z$ to $Z^{**}$ in order to make it an object of $\cC_2$, then the canonical inclusion $j\colon Z\To Z^{**}$ is a $\cC_2$-morphism. We say that  $\cC_2$ is \emph{closed under biduals} if we can do this for all $Z\in \Obj(\cC_2)$ and in this case we consider each bidual space $Z^{**}$ endowed with an extension of the operations in $Z$ which might not be unique. Examples of this are the categories of Banach spaces, Banach lattices, or $p$-convex Banach lattices (see \cite[Proposition 1.d.4]{lita} for the final assertion).

\begin{thm}\label{TheoLocCompFreeObjects}
	Suppose that we have metric-algebraic categories $$\cC_2\subseteq \cB \subseteq \cC_1,$$ $\cC_1$ is nice, $\cC_2$ is closed under biduals. Suppose that $i\colon N \To M$ is a morphism of $\cC_1$, and that $N$ and $M$ have free objects $F[N]$ and $F[M]$ in $\cC_2$. The following are equivalent:
	\begin{enumerate}
		\item $F[i]\colon F[N] \longrightarrow F[M]$ is an embedding of the Banach space $F[N]$ as a locally complemented subspace of $F[M]$.
		\item For every dual Banach space $Z^*$ endowed with a $\cC_2$-structure and every morphism $f\colon N\longrightarrow Z^*$ in $\cC_1$ there exists a morphism
		$h\colon M\longrightarrow Z^{*}$ in $\cC_1$ with $h\circ i =f$.
		\item There exists a morphism $h\colon M \To F[N]^{**}$ in $\cC_1$ such that $j\circ\delta_N = h\circ i$, where $j\colon F[N]\To F[N]^{**}$ is the canonical inclusion.
		\item There is a $\cC_2$-morphism $T\colon F[M]\To F[N]^{**}$ such that $T\circ F[i] = j$, where $j\colon F[N]\To F[N]^{**}$ is again the canonical inclusion.
	\end{enumerate}

\end{thm}

\begin{proof}
	(1)$\Rightarrow$(2). We view $Z^*$ as an object of $\cC_2$. By the definition of $F[N]$ we have a $\cC_2$-morphism $T_f\colon F[N]\To Z^*$ such that $T_f \circ \delta_N= f$ and $\|T_f\|=\|f\|$. Let $R$ be the range of $F[i]$ and let $u\colon R\To F[N]$ be the inverse operator of $F[i]\colon F[N]\To R$.  Since $R$ is locally complemented in $F[M]$, there exists an operator $\phi\colon F[M]\longrightarrow Z^{*}$ so that $\phi(y)=T_f(u(y))$ holds for every $y\in R$. We define $h\colon M\To Z^\ast$ as $h = \phi\circ \delta_M$. Then, 
	$$h\circ i = \phi \circ \delta_M \circ i = \phi \circ F[i]\circ \delta_N = T_f \circ \delta_N = f$$
	as desired. Notice that $h$ is a morphism in $\cC_1$ since $\phi$ and $\delta_M$ are $\cC_1$-morphisms.
	
	(2)$\Rightarrow$(3) is immediate by taking $Z^*=F[N]^{**}$ and $f=j\circ \delta_N$  in  condition (2).
	
	(3)$\Rightarrow$(4). 
	Take $T_h\colon F[M] \longrightarrow F[N]^{**}$ be the $\cC_2$-morphism induced by $h$ by the universal property of the free object. 
	We have that  $$T_h\circ F[i] \circ \delta_N=T_h \circ \delta_M \circ i = h \circ i = j \circ \delta_N.$$
	Applying now the uniqueness part of the universal property defining free objects, there is a unique $\cC_2$-morphism $T \colon F[N] \rightarrow F[N]^{**}$ that satisfies $T\circ \delta_N = j \circ \delta_N$. Since $j$ and  $T_h\circ F[i]$ both satisfy this, we have that 
	$T_h\circ F[i]= j$.
	
(4)$\Rightarrow$(1). First, notice that for every $x\in F[N]$, $\|x\| = \|TF[i](x)\| \leq \|T\| \|F[i](x)\|$ and this shows that $F[i]$ is an isomorphic embedding of Banach spaces. For the local complementation, just apply condition (2) of Theorem~\ref{LemmaLocalCompFreeBanach} in the category of Banach spaces, as $T$ is in particular an operator.
\end{proof}

\begin{cor}\label{theo:LocCompFree}
Under the same assumptions as in Theorem~\ref{TheoLocCompFreeObjects}, suppose that $N\subset M$ are Banach spaces. If $N$ is $\cC_1$-locally complemented in $M$, then $F[N]$ is locally complemented in $F[M]$. More precisely, if $i\colon N \rightarrow M$ is the inclusion, then $F[i]\colon F[N]\To F[M]$ is an embedding of Banach spaces with locally complemented range.	
\end{cor}

\begin{proof}
if $N$ is $\cC_1$-locally complemented in $M$, then from condition (1) in Theorem~\ref{LemmaLocalCompFreeBanach} we get condition (2) in Theorem~\ref{TheoLocCompFreeObjects} which in turn implies condition (1)  in Theorem~\ref{TheoLocCompFreeObjects}. 
\end{proof}

The previous corollary asserts that, given a Banach space $X$ and a subspace $Y\subseteq X$, if $Y$ is locally complemented in $X$ then the inclusion $i$ of $Y$ into $X$ induces an embedding $F[i]$ of $F[Y]$ into $F[X]$ with locally complemented range. As we mentioned in the introduction, in the Lipschitz-free case, this is a consequence of a result of Fakhoury. In the Banach lattice setting, we get that if $Y$ is locally complemented in $X$, then the free Banach lattice $FBL[Y]$ generated by $Y$, is locally complemented in the free Banach lattice $FBL[X]$ generated by $X$ through the natural ``inclusion''. The same applies to the free $p$-convex Banach lattices.

%

\bigskip

\noindent
\textbf{Acknowledgements.} We are deeply grateful toward the anonymous referee for carefully reading the draft and writing an exhaustive  report which enabled us to correct some mistakes and the lack of formalism of the previous version of this work.

\end{document}